\documentclass[11pt]{amsart}

\usepackage{amsmath,amssymb,amscd,amsfonts,verbatim}

\usepackage[mathscr]{eucal}

\newtheorem{thm}{Theorem}[section]
\newtheorem{lem}[thm]{Lemma}
\newtheorem{prop}[thm]{Proposition}
\newtheorem{cor}[thm]{Corollary}

\newtheorem{assu-nota}[thm]{Assumption--Notation}

\theoremstyle{remark}
\newtheorem{remark}{Remark}
\newtheorem{assu}[remark]{Assumption}
\newtheorem{example}[remark]{Example}

\newcommand{\N}{\mathbb N}
\newcommand{\pp}{\mathbb P}

\DeclareMathOperator{\Pic}{Pic}

\newcommand{\OO}{\mathcal{O}}

\numberwithin{equation}{section}
\def\Qed{\hfill\raisebox{.6ex}{\framebox[2.5mm]{}}\\[.15in]}

\title [ Even sets of $(-4)$-curves on rational surfaces ]
{ Even sets of $(-4)$-curves on rational surface
}
\author{ Mar\'ia Mart\'\i \ S\'anchez}


\begin{document}

\begin{abstract} We study rational surfaces having an even set of disjoint  $(-4)$-curves.   The properties of the surface $S$ obtained by considering the double cover branched on the even set  are studied.   It is shown, that  contrarily to what happens for even sets of $(-2)$-curves, the number
of curves in an even set of $(-4)$-curves is bounded (less or equal to 12).    The surface $S$ has always Kodaira dimension bigger or equal to zero and the cases of Kodaira dimension zero and one are completely characterized. Several examples of this situation are given.
\medskip

\end{abstract}
\maketitle

\section { Introduction}

Let $X$ be a smooth  surface.
A set of $\nu$ disjoint smooth rational curves $N_1,\ldots,N_{\nu}$ is called
an {\sl even set} if there exists $L\in \Pic(X)$ such that
 $2L\equiv N_1+\cdots+N_{\nu}$.  In this note we study  even sets of curves $N_1,\ldots,N_{\nu}$  where each $N_i$ is a $(-4)$-curve (i.e. a smooth rational curve with self-intersection $-4$) on rational surfaces.  We prove that, contrarily to what happens  for even sets of $(-2)$-curves (cf. \cite{dolga-men-par}), the number of curves in an even set of $(-4)$-curves is bounded.  More precisely we show that the maximal number  of curves in such a set is 12.

 Given an even set of smooth rational curves one can consider the double cover branched on these curves. For even sets of $(-2)$-curves on rational surfaces, such a double cover is again a rational surface (see \cite{dolga-men-par}). In contrast again  the double cover of a rational surface  branched on an even set of $(-4)$-curves has always Kodaira dimension $\geq 0$.  In this paper we characterize completely the even sets of $(-4)$-curves on rational surfaces, such that the corresponding double cover has Kodaira dimension 0 or 1.  More precisely we show that any even set of $(-4)$-curves  on a rational surface, whose corresponding double cover has Kodaira dimension 0 or 1, are components of fibres of a not relatively minimal elliptic fibration.
  We give examples for all the possible numbers of the $(-4)$-curves when the Kodaira dimension is 0. We do not know any examples for which the Kodaira dimension of the double cover is 2 and we conjecture this should not occur.

\

\noindent {\it Notation.}
 We work over the complex numbers. All varieties are projective algebraic. All the notation we use is standard in algebraic geometry.
We just recall the definition of the numerical invariants of a smooth surface $X$:  the self-intersection number $K^2_X$ of the canonical divisor $K_X$, the {\em geometric genus} $p_g(X):=h^0(K_X)=h^2(\OO_X)$, the {\em irregularity} $q(X):=h^0(\Omega^1_X)=h^1(\OO_X)$ and the {\em holomorphic Euler characteristic} $\chi(X):=1+p_g(X)-q(X)$.

A {\em $(-r)$-curve} on a surface $X$ is a smooth irreducible rational curve with self-intersection $-r$.
An {\em even set of $(-r)$-curves } is a disjoint union of $(-r)$-curves $C_1,...,C_n$ such that the divisor $C_1+...+C_n$ is divisible by 2 in $\Pic(X)$.

We do not distinguish between line bundles and divisors on a smooth
variety. Linear equivalence is denoted by $\equiv$ and numerical equivalence by $\sim$.

\medskip

\

\noindent {\it Acknowledgments.}
I am very grateful to  M. Mendes
Lopes for all her help.

The author is a collaborator of the Center for Mathematical Analysis, Geometry
and Dynamical Systems of Instituto Superior T´ecnico, Universidade
T´ecnica de Lisboa and was partially supported by FCT (Portugal) through program POCTI/FEDER,  Project
POCTI/MAT/44068/2002 and the doctoral grant SFRH/BD/17596/2004.

\

\section{General facts }

\

Throughout this section we make the following

\begin{assu}\label{hyp} $X$ is a smooth  projective rational surface and $C_1,...,C_n$ is an even set of
disjoint $(-4)$-curves on $X$. We denote by $L$ the divisor satisfying $C\equiv 2L$, where $C:=C_1+...+C_n$.
 \end{assu}

From now on to the end of this chapter, we denote by $X$ a surface in the conditions of Assumption \ref{hyp}.

 \begin{remark}\label{remark1}
 We can contract the curves $C_i$ obtaining a rational surface with $n$ quotient
 singularities of type $\frac{1}{4}(1,1)$.
 \end{remark}

\begin{prop}\label{propriedades} The divisor $L$ satisfies the following:

\begin{enumerate}

\item $h^0(X, L)=0$;
\item $K_XL+L^2=0$;
\item $(K_X+L)^2= K_X^2+n$.
\item  $1\leq h^0(X, 2K_X+C)\leq n$.
\item $1\leq h^0(X, K_X+L)\leq n$.

\end{enumerate}
\end{prop}
\begin{proof}
Assertion (i) is obvious, because $2L\equiv C$, $h^0(X,C)=1$ and $C$  is
reduced. By assumption \ref{hyp}, we have $L^2=-n$ and $K_XL=n$,
this proves (ii) and (iii). Finally, by the Riemann-Roch Theorem, one has $h^0(X, K_X+L)\geq 1$ and thus the left side  of inequalities (iv) and (v). On the other hand
 by the long exact sequence obtained from
the exact sequence:
$$0 \to \mathcal{O}_X(2K_X) \to \mathcal{O}_X(2K_X+C) \to
\mathcal{O}_C((2K_X+C)\mid_C) \to 0,$$  one has $h^0(X, 2K_X+C)\leq n$ (and so also  $h^0(X, K_X+L)\leq n$)
because  $X$ is rational, $\mathcal{O}_C(2K_X+C)=\mathcal{O}_C$ and $h^0(C,\mathcal{O}_C)=n$.

\end{proof}
\begin{remark}\label{nef}

In what follows, we can assume that  $K_X + L$ is nef. Otherwise, since $K_X+L$
is effective, there is an irreducible curve $E$ such that $E^2< 0$
and $E(K_X+L)<0$. Since each curve $C_i$ satisfies $C_i(K_X+L)=0$,
$E$ is not one of the curves $C_i$.  So $EL\geq 0$ and thus  $EK_X<0$. Since $E$ is irreducible and $E^2<0$, the only possibility is that
$E$ is a (-1)-curve disjoint from $C$ and  so we can contract it, without changing the initial assumptions.
\end{remark}

\begin{remark}\label{remark2.1}
Nefness of $K_X+L$ implies that for each $(-1)$-curve $\theta$, there exists at least one $(-4)$-curve $C_i$ such that $\theta C_i>0$.

 \end{remark}

Our next goal is to describe the double cover of $X$ branched along $C=2L$.
Let $$\pi: S \to X$$ be a double cover branched along $C$. Then $S$ is a smooth surface and
by the double cover formulas (\cite{bpv}), we have
\begin{itemize}
\item $K_S= \pi^*(K_X + L)$;
\item $\mathcal{X}(\mathcal{O}_S)=2.$
\end{itemize}

\begin{remark}

The surface $S$, having $\mathcal{X}(\mathcal{O}_S)=2$ has Kodaira dimension $\geq 0$.
Since we are assuming that $K_X+L$ is nef and $K_S=\pi^*(K_X+L)$, also $K_S$ is nef and thus  $S$ is  minimal.
\end{remark}

\begin{lem}\label{lem-2KX}
Let $X$ be a rational surface with an even set $C$ of $(-4)$-curves. Then  $h^0(X, -2K_X)\leq 1$. Furthermore   $h^0(X, -2K_X)\neq 0$ if and only if the double cover $S\to X$
is  a  $K3$ surface.
\end{lem}
\begin{proof}
Notice that $h^0(X,C)=1$ and $h^0(X, 2K_X +C)\geq
 1$ by Proposition \ref{propriedades}. Thus we conclude that either $h^0(-2K_X)=0$ or
 $h^0(X, -2K_X)=1$ and
 $h^0(X, 2K_X+C)=1$. Since $K_XC_i=2$, if $h^0(-2K_X)=1$, then each $C_i$ is
 a component of $-2K_X$. So we can write $-2K_X= C + \Gamma$, where $\Gamma $ is an effective divisor and letting $\Delta= 2K_X +
 C$ we obtain $\Delta + C + \Gamma\equiv C$.  Hence $\Delta=\Gamma=0$
 namely $C=-2K_X$. Since $X$, being rational, has no 2-torsion,   also $L=-K_X$ and so $K_S=\OO_S$.  Thus $S$ having $p_g=1$, $\chi=2$  is a $K3$ surface.
Conversely, if $S$ is a $K3$ surface, $K_S=\OO_S$ and the result follows.
\end{proof}

Next, we apply the above results to the following proposition.
\begin{prop}\label{classification}
Let $X$ be a rational surface with an even set $C$ of  $n$ $(-4)$-curves and such that $K_X+L$ is nef.  Then  $-n\leq K_X^2\leq -1$.  Furthermore

\begin{enumerate}
\item  if $K_X^2=-n$, then
    \begin{itemize}
    \item[($i_a$)] $\kappa(S) = 0$ $\Rightarrow$ $S$ is a $K3$ surface, or
    \item[($i_b$)] $\kappa(S) = 1$ $\Rightarrow$ $S$ is an elliptic
    surface;
    \end{itemize}
\item if $K_X^2 > -n$, then $K_S ^2 \ge 2$ and $S$ is a surface of general type.
\end{enumerate}
\end{prop}
\begin{proof}

Since $K_X+L$ is effective and nef one has $(K_X +L)^2\geq 0$. So by Proposition \ref{propriedades} (iii)  $K_X^2\geq -n$. As we
have seen, one has $h^0(X, K_X+L)\geq 1$ and so since $h^0(X, L)$ must be $0$ we
 conclude that $h^0(-K_X)=0$, otherwise  the map

 $$H^0(X,-K_X)\otimes  H^0(X, K_X+L)\to H^0(X,L)$$
 would have nonzero image. So by the Riemann-Roch Theorem necessarily $K_X^2\leq -1$.

 The rest of the proposition is clear, by
 the classification of  minimal surfaces (see e.g.
\cite{beauville1}) and $K_S^2=2(K_X^2+n)$.
\end{proof}

Finally we recall  an important result due to Miyaoka.
\begin{prop}\cite{Miya}\label{Miyaoka}
The number $n$ of disjoint $(-2)$-curves on a
surface $W$ with $K_W$ nef satisfies $3c_2(W)-K_W^2\geq \frac{9}{2}n$.
\end{prop}

With this result we obtain:

\begin{lem}
If  $X$ is a rational surface with an even set $C$ of  $n$ $(-4)$-curves, then
$$n\leq 16.$$
Furthermore, if equality holds, then $K_S^2=0.$
\end{lem}
\begin{proof}
Since  $\mathcal{X}(\mathcal{O}_S)=2$, one has $c_2(S)\leq 24$. On
the other hand, for each curve $C_i$, $\pi^{-1}(C_i)$ is a $(-2)$-curve
in $S$. Finally, applying Proposition \ref{Miyaoka} we obtain the result.
\end{proof}

 We will start by studying the case $K_X^2=-n$, this is the K3 and the elliptic case.

\


\section{The elliptic fibration}


\

In this section we want to prove the following:

\begin{prop}\label{thm_elliptic}
Let $X$ be as in Assumption \ref{hyp} and let $S$ be the associated double cover of $X$.  Then $\kappa(S)\leq 1$ if and only if $X$ has an elliptic fibration containing the $(-4)$-curves.
\end{prop}

Before proving the above result, we will need various facts that we now establish.

Suppose that $X$ has an elliptic fibration
$$p': X \to \pp^1$$
with general fibre $F'$ such that every  $(-4)$-curve is contained in a fibre.

\

Since  $F'C_i=0$ and $F'C=0$,   $\mathcal{O}_{F'}(C)=\mathcal{O}_{F'}$ and so either
$\mathcal{O}_{F'}(L)=\mathcal{O}_{F'}$
 or
$\mathcal{O}_{F'}(L)\ne \mathcal{O}_{F'}$.

In the first case for a general fibre $F'$, $\pi^*(F')$ is disconnected. More precisely  $\pi^*(F')$  is the union of two fibres of an elliptic fibration on $S$. In the second case $\pi^*(F')$  is connected and, by the Hurwitz formula, again  elliptic.

   So we have
\begin{lem}\label{elliptic}
With the above notation, $\pi^*F'$ gives an elliptic fibration on $S$ and
we have the following
    commutative diagram:
$$
\begin{CD}\ S@>p>>B\\ @V\pi VV  @VV \pi' V\\ X@>p' >>\mathbb{P}^1
\end{CD}
$$
\end{lem}

\

Moreover,
\begin{lem}
In the above situation one of the following holds
\begin{itemize}
    \item $h^0(-2K_X)=1$, $\kappa(S)=0$ and $S$ is a K3 surface
    \item $h^0(-2K_X)=0$, $\kappa(S)=1$ and $S$ is an elliptic surface.
\end{itemize}
\end{lem}
\begin{proof}
The proof follows by Proposition \ref{classification} and Lemmas \ref{lem-2KX} and \ref{elliptic}.
\end{proof}

Having established the properties above we now examine the converse situation. Let $X$ be as in Assumption \ref{hyp} and $S$ the double cover of $X$ branched in $C$.

If $\kappa(S)=0$, then $S$ is a K3 surface, so $K_X+L\equiv 0$ and
 $X$ is a Coble surface, this is, a nonsingular rational surface with empty anticanonical linear system $|-K_X|$ but nonempty bi-anticanonical system $|-2K_X|$. By the results in \cite{dolga-zhang}, in this case, there is a birational
morphism $X \to \mathbb{P}^2$ such that the image of  $C \in
|-2K_X|$ in $\mathbb{P}^2$ is a member of $|-2K_{\mathbb{P}^2}|$, whence a plane
sextic, that will be called a Coble sextic.

\begin{lem}\label{rem_coble_sextic}
 The classification in (\cite{dolga-zhang}, Section 5) yields:
\begin{itemize}
    \item [(i)] If $n=1$, the image of the irreducible $(-4)$-curve $C$ on
    $\mathbb{P}^2$ is an irreducible member of a pencil of sextics with nine distinct double base points and having an extra singular double
    point.
    \item [(ii)] If $n>1$,
    then the image of $C$ is the union of two members
    (both singular) of a pencil of cubics.
    \end{itemize}
\end{lem}

\

We are ready to prove Proposition \ref{thm_elliptic}\\
{\it Proof of Proposition \ref{thm_elliptic}.}\
If $X$ has an elliptic fibration containing the $(-4)$-curves the result follows by Lemma \ref{elliptic}.

Conversely, assume that $\kappa(S)\leq 1$.
By Proposition \ref{classification}, we have two possibilities: $\kappa(S)=0$ or $1$.

For $\kappa(S)=0$, $S$ is a K3 surface and applying Lemma \ref{rem_coble_sextic} we obtain the result. In fact, if $n=1$, then $X$ is
a rational elliptic surface with one multiple fibre
of multiplicity $2$ whilst, if $n> 1$, $X$ is a rational elliptic surface
without multiple fibres.

For $\kappa(S)=1$,
 $S$ is an elliptic surface and so there is a smooth curve $B$ and a
surjective morphism $p:S\to B$ whose generic fibre is a nonsingular elliptic curve $F$.
It is well known that: $\mathcal{X}(\mathcal{O}_S)=2$ implies
\begin{equation}\label{eq-elliptic}
K_S \sim 2g(B)F + \sum_{i=1}^{r} \frac{m_i-1}{m_i}F,
\end{equation}
where $m_1F_1,...,m_rF_r$ are the multiple fibres of $p:S\to B$.\\
Note that $K_S\pi^*C_i=2(K_X+L)C_i=0$. Since
the elliptic fibration of $S$ must be invariant under the involution associated to the double cover $\pi: S \to X$, it induces a fibration of $X$, $p':X \to \mathbb{P}^1$, whose general fibre we denote by $F'$.
 Since $F\pi^*C_i=0$, also $F'C_i=0$ and $F'C=0$. Using the same reasoning as in Lemma \ref{elliptic}, we see that also $p':X \to \mathbb{P}^1$ is an elliptic fibration and we are done.
\Qed

\begin{remark}\label{halphen}
{\rm Recall that a \textit{Halphen pencil of index $m$} is an irreducible pencil of plane
 curves of degree $3m$ with 9  base points  of multiplicity $m$ (some of them may
 be infinitely near). By \cite{cossec-dolga}, the minimal resolution of a Halphen pencil of index $m$
 is a rational elliptic surface with a multiple fibre of multiplicity
 $m$.

 Conversely, again by   Theorem 5.6.1 in \cite{cossec-dolga}, if $f:X'\to \mathbb{P}^1$ is a rational elliptic surface with a multiple fibre of multiplicity $m$ ($m=1$ if
 it does not have multiple fibre), then
 there exists a birational morphism $\tau: X' \to \mathbb{P}^2$ such
 that the composition of rational maps $f\circ
 \tau^{-1}:\mathbb{P}^2 \dashrightarrow \mathbb{P}^1$ is given by a Halphen
 pencil of index $m$.

So, if $X$ is a surface as in Assumption \ref{hyp}, then $\kappa(S)\leq 1$ if and only if  there is a Halphen pencil of index $m$ in $\pp^2$ corresponding to the elliptic fibration. In particular for $\kappa(S)=0$, $m=1$ or $m=2$.}
 \end{remark}

 \

\section{The $K3$ case}


\

In this section we assume that $K_S \equiv \mathcal{O}_S$, thus
 one has $-K_X\equiv L$ and $-2K_X\equiv C$.

Notice that $S$
is a smooth K3 surface with an involution $\sigma$ such that
$\sigma^*\omega=-\omega$ for a nonzero holomorphic 2-form. Zhang in \cite{zhang} classified the quotients of $K3$ surfaces modulo involutions. In particular, with the approach top-down, he proved the following

 \begin{prop}[\cite{zhang}] \label{teorema}
Assuming that $X$ is a rational surface with an even set of $n$ disjoint $(-4)$-curves. If the double cover $S$
 is a $K3$ surface, then we can conclude that
$$1\leq n\leq10,$$ and $n$ can take any value in this range (see examples below).
\end{prop}
 \

\begin{remark}
It is possible to give another proof of the previous proposition using similar arguments to those used in the following sections.
\end{remark}

\begin{example}
Let us give an example of each possible case:
\begin{enumerate}
\item For $n=1$, as we have seen above, take a sextic curve  in $\mathbb{P}^2$
with ten double
points. After blowing them up, we obtain a rational surface with a
$(-4)$-curve. See also (\cite{zhang}, Example 2.7).
\item For $n=2$, in a pencil of cubics of $\mathbb{P}^2$, take  two
singular members with one node each, $N_1$ and $N_2$ respectively. Blowing up the base
points, $N_1$ and $N_2$ we obtain the two $(-4)$-curves. This example,
with another point of view, can be found in \cite{lee-park}.
\item For $n=3$, in a pencil of cubics of $\mathbb{P}^2$ take  two
singular members: one cubic with a node and a conic plus a
line. Blowing up the base
points and the singular points of these two singular members we obtain
the result.
\item For $n=4$, take a pencil of cubics and choose two singular cubics:
three nonconcurrent
 lines $L_1,L_2,L_3$ and one
 cubic, $C_1$, with one
unique singular point, a
double point $N$.
Let $$C=L_1+L_2+L_3+C_1$$
 with the following singular points:

$L_i\cap L_j=P_{ij}$ with $i,j=1,2,3$, and
$L_1\cap C_1=\{Q_1,Q_2,Q_3\}$, $L_2\cap C_1=\{R_1,R_2,R_3\}$, $L_3\cap C_1=\{S_1,S_2,S_3\}$ and $N$.

Blowing up the thirteen singular points of $C$ we obtain $p: X \to \mathbb{P}^2$ with
$\tilde{L}_i^2=-4$, $\tilde{C}_1^2=-4$ and rational.
\item For $n=5$, take a pencil of cubics and $C$ is the sum of two singular members one conic plus a line,
 and three lines.
\item For $n=6$, Let $L_1,...,L_6$ be six nonconcurrent lines in $\mathbb{P}^2$
and $C=L_1 +...+L_6\sim 6H$. Then $C$ has 15 singular points, we blow up
each line in 5 different points $p:X \to \mathbb{P}^2$,
then $\widetilde{L_i}^2=-4$, this is, take a pencil of cubics with two singular cubics $L_1+L_2+L_3$ and
$L_4+L_5+L_6$.

\item For $n=7$, in $\mathbb{F}_0$ denote by $L_1$ and $L_2$ the two
rulings, then take seven effective divisors $R_1,R_2, R_3, M_1, M_2, M_3$ and $S$ where $R_i\sim L_1$, $M_i\sim L_2$ with $i=1,2,3$,
and finally $S\sim L_1+ L_2$ without singular points. Blowing up the
intersection points we obtain seven $(-4)$-curves and the double cover is a $K3$ surface.
This is, in $\mathbb{P}^2$ take a pencil of cubics and the two singular cubics
are one nonsingular conic $C$ plus a line $L_1$ and three nonconcurrent
lines $L_2,L_3, L_4$ with these intersection points $C\cap L_i=P_{i1}$ and $P_{i2}$
such that $L_1\cap L_2=P_{11}=P_{21}$ and $L_3\cap L_4=P_{31}=P_{41}$, the
other intersection points are all different.

\item For $n=8$, in $\mathbb{F}_0$, with the same notation as above,
take $R_1\sim R_2\sim R_3\sim R_4 \sim L_1$ and
$M_1\sim M_2\sim M_3\sim M_4 \sim L_2$. As before, in $\mathbb{P}^2$
we take a pencil of cubics with these two singular members: $L_1 + L_2 +
L_3$, three concurrent lines, write $L_1\cap L_2\cap L_3=P$ and $L_3 + L_4 + L_5$ three concurrent
lines as well, write $L_3\cap L_4\cap L_5=Q$, such that $P \neq Q$ and the
other intersection points are all different.
\item For $n=9,10$, in $\mathbb{P}^2$ take six lines $L_1,...,L_6$
with this configuration: $L_1,L_2$ and $L_3$ with a common point $P$;
$L_1,L_4$ and $L_5$ with a common point $Q_1$; $L_2,L_4$ and $L_6$ with a common point
$Q_2$; $L_3,L_5$ and $L_6$ with a common point $Q_3$; finally $L_1\cap L_6=R_1$, $L_2\cap L_5=R_2$
and $L_3\cap L_4=R_3$, all different points. First of all, we
blow up $P$, then  $\tilde{L}_1, \tilde{L}_2$,
$\tilde{L}_3$ are fibres of $\mathbb{F}_1$ and denote by $L$ the exceptional curve lying over the point $P$;
now we blow up each fibre six times and we
obtain ten $(-4)$-curves, two of them in each fibre, plus the strict transform of
$\tilde{L}_1, \tilde{L}_2, \tilde{L}_3$ and $L$.

 We can obtain $n=9$  in a similar way, see for instance Example 2.10 of
 \cite{dolga-zhang}.

\end{enumerate}
\end{example}

\


\section{The elliptic case}

\

Now we examine more closely the case when
$S$ is an elliptic surface with $\kappa(S) = 1$.
Let $F$ be a general fibre of the elliptic fibration
 $p:S\to B$.

Since the Kodaira dimension of $S$ is 1 there are  effective nonzero n-canonical
divisors. These are supported on the fibres of the elliptic pencil (see equation \ref{eq-elliptic}) and so
the elliptic pencil is unique. Thus it is necessarily invariant under the involution associated to the double cover and we have a commutative diagram:\begin{equation}\label{eq_elliptic}
\begin{CD}\ S@>p>>B\\ @V\pi VV  @VV \pi' V\\ X@>p' >>\mathbb{P}^1
\end{CD}
\end{equation}
\medskip
where $p': X \to \pp^1$ is also an elliptic fibration.

We want to prove:

\begin{thm}
In the above situation we have:
\begin{itemize}
\item[(i)] $1 \leq n \leq 12$;
\item[(ii)] if $F'$ is the general fibre of $p'$, then  $\pi^*(F')=F_1 + F_2$ is disconnected.
\end{itemize}
\end{thm}

The proof will be given throughout this section.\\

As we have seen in the proof of Proposition \ref{thm_elliptic}, we can write
$$
K_S \sim 2g(B)F + \sum_{i=1}^{r} \frac{m_i-1}{m_i}F,
$$
where $m_1F_1,...,m_rF_r$ are the multiple fibres of $p:S\to B$.
\

First of all we are going to describe the singular fibres of $p'$ containing
$(-4)$-curves.

\begin{prop}\label{prop-ellip}
There exists a birational morphism $\epsilon:X \to X'$, where $X'$ is a relatively minimal elliptic surface
 with the following property:

for $i=1,...,n$, the curve $C_i':=\epsilon (C_i)$ is an irreducible component
of a fibre of type  $II, III, IV$ or $mI_r$ with $1\leq r \leq 9$ and $m\geq 1$.
\end{prop}

\begin{proof}
Notice first that, since $K_X^2=-n$, $X$ is not relatively minimal and so there exists
a $(-1)$-curve $\theta_1$ such that $\theta_1$ is a component of one of the fibres  $F'$ of the elliptic fibration of $X$.
Besides, by Remark \ref{remark2.1}, there exists a $(-4)$-curve $C_1$,
with $\theta_1C_1 = \alpha \geq 1.$
As we have seen above $C_1 $ is contained in some fibre $ F'$ and so also $(\theta_1 + C_1)\subset
F'$. As a consequence, by Zariski's lemma
$$(\theta_1 + C_1)^2= -1 +2 \alpha -4 \leq 0$$
 and so $\alpha\leq 2.$

 \

If $\alpha =2$,  $2\theta_1 +C_1\subseteq F'$ and $(2\theta_1
+C_1)^2=0$ mean that   $m(2\theta_1 + C_1)=F'$ for some $m\in \N, m\geq 1$. Contracting
$\theta_1$, $\epsilon_1:X\to X_1$, we obtain a fibre of type    $I_1$, $mI_1$ or $ II $ in $X_1$ and
$K^2_{X_1}=-n+1$.

\

If $\alpha =1$, since $\theta_1 L \geq 1$, one has $\theta_1 C\geq 2$
and so $\theta_1$ meets, at
least, another $(-4)$-curve $C_2$. As in the precedent paragraph we conclude that
 $\theta_1 C_2=1$.
Let  $\epsilon_1:X \to X_1$ be the
blowing down of $\theta_1$ and $\epsilon_1(C_i)=\tilde{C}_i$. Then $\widetilde{C}_1$ and $\widetilde{C}_2$ are curves
with self-intersection $-3$ and $\widetilde{C}_1 \widetilde{C}_2=1$. As
before, we have in this fibre  a $(-1)$-curve $\widetilde{\theta}_2$.
Hence either $\widetilde{\theta}_2$ comes from a $(-2)$-curve in $S$, or $\widetilde{\theta}_2$ comes from a $(-1)$-curve  in $S$.

If $\widetilde{\theta}_2$ comes from a $(-2)$-curve $\theta_2$, then  it is easy to see that  $4\theta_1 + 2\theta_2 +C_1 + C_2$ is a fibre. Contracting $\theta_1$ and
    $\theta_2$ we obtain a fibre of type $III$.

If $\widetilde{\theta}_2$ comes from a $(-1)$-curve, there are three possibilities:
\begin{enumerate}
  \item      $\theta_2 C_1=\theta_2 C_2=1$. Then we have
    $(2\theta_1 + 2\theta_2 +C_1+C_2)^2=0$  and so we have a  fibre or a rational multiple of a fibre . Contracting $\theta_1$
    and $\theta_2$, we obtain that an integer multiple of the image of $2\theta_1 + 2\theta_2 +C_1+C_2$ is a fibre of type $I_2$.

  \item  $\theta_2C_1=\theta_2C_2=0$.
  First of all, let us point out that 
 every $(-4)$-curve  $C_i$ meets, at most, three  $(-1)$-curves (possibly infinitely near).
    In fact if $C_i$ meets  four   $(-1)$-curves, it is not very difficult to see that contracting these the image of $C_i$ is a smooth rational curve  with self-intersection $0$. This is impossible because the fibres of an elliptic fibration have always $p_a=1$. Thus,  since
    $\widetilde{C}_1^2=-3$, there exists  a $(-2)$-curve   $\delta$ with
    $\delta
    C_1=1$ and $\delta \theta_2=1$.
    There are also two, and only two,
    $(-4)$-curves $C_3$ and $C_4$
   intersecting $\theta_2$, this is $\theta_2 C_3=\theta_2 C_4=1$. Then we have
    $C_3 + C_4 + 4\theta_2 + 2\delta + C_1+ ...$ in the fibre,
    but as we have seen before
    $C_3 + C_4 + 4\theta_2 + 2\delta$ is a fibre, and we obtain a contradiction.

   \item  $\theta_2 C_2=0$ and $\theta_2 C_1=1$.

    If there is another $(-1)$-curve $\theta_2'$ such that $\theta_2' C_1=1$,
     we have  $\theta_2' C_2=0$; if not, $\theta_2' C_2=1$ and then
     $(2\theta_1 + 2\theta_2' + C_1 + C_2)^2=0$, but we have
     $2\theta_1 + 2\theta_2' + C_1 + C_2 + \theta_2\subseteq f$,
     which is absurd.
     We have seen above that there exist $C_3$ and $C_4$,
    different $(-4)$-curves, such that $\theta_2 C_3=1$ and $\theta_2'C_4=1$.
    Then we have $4\theta_1 + 4\theta_2 + 4\theta_2' + 3C_1+ C_2+ C_3 + C_4 \subseteq
    f$ and
    $(4\theta_1 + 4\theta_2 + 4\theta_2' + 3C_1+ C_2+ C_3 + C_4)^2=0$, we get a
    fibre of type $IV$.

     If there is not another $(-1)$-curve intersecting $C_1$ and since $\theta_2
    C_2=0$, there exists another $(-4)$-curve $C_3$ with $\theta_2
    C_3=1$. As before, blowing down $\theta_1$ and $\theta_2$,
    $\widetilde{C}_3$ and $\widetilde{C}_2$ are curves with
    self-intersection $-3$, so there exists another
     $(-1)$-curve $\theta_3$. Since, as before, $\theta_3 C_3=1$, then either $\theta_3
    C_2=1$ or $\theta_3 C_2=0$. If $\theta_3 C_2=1$ we
    have $2\theta_1 + 2\theta_2 + 2\theta_3 + C_1 +C_2 + C_3 \subseteq
    f$ and $(2\theta_1 + 2\theta_2 + 2\theta_3 + C_1 +C_2 +
    C_3)^2=0$, so we
    obtain a fibre
    (or a rational multiple of a fibre)
    of type $I_3$. In the other case, $\theta_3 C_2=0$, there is another $(-4)$-curve
    $C_4$ ..., repeating the same argument we obtain a fibre of
    type $mI_r, r\geq 4$ and $m\geq 1$.

\end{enumerate}
    In conclusion, since $n$ is a finite number and $K_X^2=-n$ there are  $n$
     $(-1)$-curves  in the fibres
    and after contracting them by  $\epsilon: X \to X',$ one obtains a relatively minimal rational elliptic
     surface $X'$ with these singular fibres.

    By \cite{har-lang}
    a connected fibre on a relatively minimal rational elliptic surface
    has at most nine irreducible components, and so in particular   for $mI_r$ we have $1\leq r \leq 9$.
\end{proof}

 Denote by $F'_j$ the elliptic fibres
 in $X$ containing $(-4)$-curves and by $$\mathcal{J}=\{F'_j, j=1,...,n'\}$$ the set of these
 fibres.
 Also denote by
$\mathcal{J'}=\{\epsilon(F'_j), j=1,...,n'\}$ the image of $\mathcal{J}$ in $X'$.

\

Keeping this notation:
 \begin{cor}
 The number $n$ of $(-4)$-curves in $C$ is at most $12$.
 In particular, if $n=12$ the singular fibres of the elliptic fibration of $X'$ are
 all in $J'$.
 \end{cor}
\begin{proof}
 Since $X'$ is a relatively minimal rational elliptic surface we have that
 $c_2(X')=12$ and by (\cite{beauville1}, Lemma VI.4), we know that
 $$c_2(X')=\sum_{s}\mathcal{X}_{top}(F'_s),$$
 with $F'_s$ the singular
fibres.

Also, noticing that $\mathcal{X}_{top}(I_n)=n$, $\mathcal{X}_{top}(II)=2, \ \mathcal{X}_{top}(III)=3$ and
 $\mathcal{X}_{top}(IV)=4$,  the result follows by Proposition \ref{prop-ellip}.
 \end{proof}

\begin{remark}\label{obs3.1}
 {\rm The pull-back $\pi^*(F'_j)$, where $F'_j\in \mathcal{J}$, will be one of following types:
 \begin{itemize}
\item If $\epsilon(F'_j)$ is of type $mI_r$, this is $F'_j=m(\sum_1^{r} C_i +
 2\theta_i)$,
 then $\pi^*(F'_j)=m(\sum_1^{r} 2\gamma_i + 2\hat{\theta}_i)=2m(\sum_1^{r} \gamma_i +
 \hat{\theta}_i)$, therefore  $\pi^*(F'_j)=2mI_{2r}$ with $1\leq r\leq 9$ and $m\geq 1$.

 \item If $\epsilon(F'_j)$ is of type $II$, then $\pi^*(F'_j)=\pi^*(2\theta_1 +
 C_1)=2(\hat{\theta}_1+\gamma_1)$, therefore $\pi^*(F'_j)=2III$.

\item If $\epsilon(F'_j)$ is of type $III$, then $F'_j=4\theta_1 +
2\theta_2 +C_1 +
 C_2$ with $\theta_2$ a $(-2)$-curve such that $\theta_2C=0$, hence
$\pi^*(F'_j)=2(2\hat{\theta}_1 +\theta^1_2+\theta^2_2+ \gamma_1+ \gamma_2)$, so $\pi^*(F'_j)=2\tilde{D}_4$.

\item Finally, if $\epsilon(F'_j)$ is of type $IV$, then
$\pi^*(F'_j)=2(2\hat{\theta}_1+ 2\hat{\theta}_2 + 2\hat{\theta}_3 + 3\gamma_1+ \gamma_2 + \gamma_3)$, so $\pi^*(F'_j)=2\tilde{E}_6$.
 \end{itemize} }
 \end{remark}

 \

 We want now to understand $\pi^*(F')$. So,
  we begin by supposing  that $\pi^*(F')=F$.  In this case $B=\pp^1$
 and we can consider the commutative diagram:
$$
\begin{CD}\ S@>p>>\mathbb{P}^1\\ @V\pi VV  @VV \pi' V\\ X@>p' >>\mathbb{P}^1
\end{CD}
$$

 \



\begin{lem}
If $\pi^*(F')=F$,  every fibre   $\epsilon(F'_j)\in \mathcal{J'}$ is a fibre of type
$mI_r$. In particular,
$$F'_j= m_j (\sum_1^{r_j} C_i + 2\theta_i),$$
with $n$  $(-1)$-curves $\theta_i$, $1\leq r_j\leq 9$ and $m_j\geq 1$.
\end{lem}
\begin{proof}
If $\pi^*(F')=F$, by Remark \ref{obs3.1} the pull-back  of any fibre containing a $(-4)$-curve is a double fibre of the elliptic fibration in $S$.
Since for every multiple fibre $mF_0$ in a elliptic fibration $F_0$
cannot be simply-connected (cf. \cite{bpv}), looking at the description of the fibres in Proposition \ref{prop-ellip}  we obtain the statement.
\end{proof}

It is well known (see \cite{cossec-dolga}) that every relatively minimal rational elliptic surface has at most
one multiple fibre.

We analyse the different possibilities for the multiple fibres to
prove the next proposition:

\begin{prop}
The elliptic fibration $F'$ of  $p':X \to \pp^1$, diagram (\ref{eq_elliptic}),
satisfies $$\pi^*(F')=F_1 + F_2.$$
\end{prop}

\begin{proof}
Under the assumption that $\pi^*(F')=F$,
let $mD$ be the unique multiple fibre in $X'$, if it has any, otherwise let $m=1$
and $D$ be any fibre.


 First of all assume that $mD$ $\notin$  $\mathcal{J'}$, then $m_j=1$ in $F'_j$ for all $j=1...n'$. The
 multiple fibres in $S$ are $\pi^*(mD)$ of multiplicity $m$, and
 $\pi^*(F'_j)$
 of multiplicity $2$,
 $j=1,...,n'$. Thus, since $\mathcal{X}(\mathcal{O}_S)=2, $  and the elliptic fibration has base $\pp^1$,

  $$K_S \equiv \frac{(m-1)}{m}F + \frac{n'}{2}F.$$
  On the other hand, since
  $$K_X \equiv -F' + \frac{(m-1)}{m}F'+ \theta_1 + ... + \theta_n,$$
  by
  the double cover formulas we obtain
  $$K_S\equiv -F + \frac{(m-1)}{m}F + \widehat{\theta}_1 + ... + \widehat{\theta}_n + \gamma_1 +...
  \gamma_n=\frac{(m-1)}{m}F
    +(\frac{n'}{2}-1)F,$$ a contradiction.

 Now, assume that $m>1$ and $mD \in \mathcal{J'}$. Then the multiple fibres in
 $S$ are $\pi^*(mD)$ of multiplicity $2m$ and $\pi^*(F'_j)$ of multiplicity $2$, with $j=1,...,n'-1$.
 Thus
 $$K_S \equiv \frac{(2m-1)}{2m}F + \frac{(n'-1)}{2}F.$$
 As before, $K_X\equiv -F' + \frac{(m-1)}{m}F'+ \theta_1 + ... + \theta_n$ and by
 the double cover formulas we obtain
$$K_S\equiv -F + \frac{(m-1)}{m}F + \widehat{\theta}_1 + ... +
\widehat{\theta}_n + \gamma_1 +...
  \gamma_n=$$
  $$ =\frac{(m-1)}{m}F
    +(\frac{(n'-1)}{2}-1)F + \frac{1}{2m}F,$$
    since $m$ is a natural number we have a contradiction again and the
    result follows.
    \end{proof}

\

 So $\pi^*(F')$ is disconnected and
there is  a natural $2-1$ map $\pi': B\to \mathbb{P}^1$. By using the Hurwitz formula we get
 $$g:=g(B)=\frac{deg R}{2}-1,$$ where $R$ is the ramification divisor.
 Let us recall the commutative diagram (\ref{eq_elliptic})
$$
\begin{CD}\ S@>p>>B\\ @V\pi VV  @VV \pi' V\\ X@>p' >>\mathbb{P}^1
\end{CD}
$$

\

Then keeping the above notation:
\begin{cor}

One has $$n' \leq deg \, R\leq n'+1 $$
\end{cor}
\begin{proof}
 Since $S$ is branched along $C=C_1+...+C_n$, each $F'_j \in \mathcal{J}$ corresponds to
 a ramification point of $\pi'$, so $n'\leq deg \, R$. Also, since $p':X'\to \pp^1$ has at most one multiple fibre,  there is at most one more
ramification point corresponding to a multiple fibre (necessarily of even multiplicity)
 of the elliptic fibration of $X',$ hence the result.
\end{proof}

\

\begin{cor}
The singular fibres of $p$ coming from fibres $F'_j\in \mathcal{J}$ of $X$ are of
type $mI_{2r}$ ($1\leq r\leq 9$, $m\geq 1$),  $III,$ $\tilde{D}_4$ and $\tilde{E}_6$.
\end{cor}
\begin{proof}
Using  Remark \ref{obs3.1}, notice that
if $\epsilon(F'_j)$ is a fibre type $mI_r$,
 then $\pi^*(F'_j)=2m(\sum_1^{r} \gamma_i +
 \hat{\theta}_i)$. Therefore we obtain in $S$ a fibre of type $mI_{2r}$ with $1\leq r\leq 9$ and $m\geq 1$.
 If $\epsilon(F'_j)$ is type $II$, $III$, or $IV$, then
$\pi^*(F'_j)=2F$, with $F$ a
fibre of type $III,$ $\tilde{D}_4$ and $\tilde{E}_6$ respectively.
\end{proof}

\

\noindent {\it Examples.}
 \begin{itemize}
\item [-] We take a pencil
 of cubics with  $12$ nodal cubics, blowing up the base points and these $12$ double points we will obtain a rational elliptic surface
 with $K^2_X=-12$ and $n=12$.
 The double cover $S$ will be an elliptic surface. In a similar way,
  we can obtain examples
  for $n\geq 4$ even.

  \item [-] Using  the program Magma, we can find a pencil of sextics with $4$ nodal sextics. Blowing up these double points and the base points, we get examples for $n=2$, $3$ or $4$.
 \end{itemize}

\

\section{Some remarks on the case of general type}

\

Suppose  now that $S$ is a surface of general type.
From \cite{Miya1},  $K_S^2\leq 9 \mathcal{X}(\mathcal{O}_S)$  and so $K^2_S\leq
18$.

Recalling  that $K_S^2=2(K_X^2+n)$ and
$-n<K_X^2\leq -1$, so

\begin{equation}\label{eq_general_type}
2\leq K_S^2\leq 2(n-1), \quad
2\leq h^0(2K_X + C)\leq n \quad \text{and} \quad n\geq 2.
\end{equation}

\

We do not know if this case can happen unlike in the previous cases we know examples.
Below we give some properties for this situation. More precisely,
throughout this section we will prove:

\begin{prop}\label{prop_tipo_gral}
Suppose that $S$ is a surface of general type. One of the following holds:
\begin{itemize}
\item  if $S$ is regular, then
$2\leq K_S^2\leq 8$
and $2\leq n \leq 9$;
\item if $S$ is irregular, then $q(S)=1$. Also,
 $6\leq K_S^2\leq 10$ and $4\leq n\leq 9$.
\end{itemize}
\end{prop}

\begin{proof}
We divide the proof into steps.\\

\textbf{Step 1:} \textit{ $h^0(2K_X + C)=h^0(2K_X+L)= K_X^2 + n+1.$}\\
By the Riemann-Roch theorem
$$h^0(2K_X + C)= K_X^2 + n+1 + h^1(2K_X+C).$$
The
projection formula
$h^1(2K_S)=h^1(2K_X+C) + h^1(2K_X +L),$
together with
$h^1(2K_S)=0$, gives  $h^1(2K_X+C)=0$ and so $h^0(2K_X + C)= K_X^2 + n+1.$
Since $K_X+L$ is effective, nef and big, then $h^1(-K_X-L)=0$ (see \cite{reid0})  and thus
$h^0(2K_X+L)=K_X^2 +n+1$ as asserted.

\

\textbf{Step 2:} \textit{The canonical divisor satisfies $2\leq K^2_S \leq 10$ and $2\leq n\leq 14$.}\\
If $K_S^2=18$, by the Noether's formula one has $c_2(S)=6$ and applying Miyaoka's formula
(Proposition \ref{Miyaoka}),
we obtain $n=0$: also for $K_S^2=16$
 we obtain $n=1$.
Then $K_S^2\leq 14$.

Similarly, if $K_S^2=14$, then $n\leq 3$. Then we have  a contradiction by the inequalities (\ref{eq_general_type}).
The same argument proves that $K_S^2\neq 12$. Then $K_S^2\leq 10$.

If $K_S^2=10$ then $n\leq 7$. From Proposition \ref{classification}, one gets $K_X^2\leq -1$, whence
 the only possibility is $6\leq n \leq
7$. For $n=7$ we have $K_X^2=-2$ and for $n=6$ we have $K_X^2=-1$. In
the same way,
if $K_S^2=8$ then $5\leq n \leq 8$,
if $K_S^2=6$ then $4\leq n \leq
10$, if $K_S^2=4$ then $3\leq n \leq 12$, and finally if $K_S^2=2$,
then $2\leq n\leq 14$.

\

From now on, we are going to analyse separately the cases when $S$ is regular
and irregular.

\

\textbf{Step 3:} \textit{If $S$ is a regular  surface of general type, then
$2\leq K_S^2\leq 8$ and $2\leq n \leq 9$.}\\
The hypothesis  $q(S)=0$  implies
$b_2(S)=c_2(S)-2$ and $p_g(S)=1$, hence $b_2(S)=22-K_S^2$. Since $\pi: S \to X$ is an holomorphic map of
degree 2, then
$$\pi^*:H^2(X,\mathbb{R}) \to H^2(S,\mathbb{R})$$
is an
injective ring homomorphism. We have $b_2(X)=h^{1,1}(X)$ and
$$b_2(S)=h^{2,0}(S) + h^{1,1}(S) + h^{0,2}(S),$$ where
$h^{2,0}(S)=h^{0,2}(S)=1$. Since  $h^{1,1}(X)\leq h^{1,1}(S)$, one has
$$ b_2(X) \leq b_2(S) - 2.$$

Note that, since $X$ is a rational surface, $b_2(X)= 10-K_X^2$.
If $K^2_S=2$, then  $b_2(S)=20$ and that implies $K^2_X\geq
-8$.  In conclusion, $K^2_X\geq -8$ and from $K_S^2=2(K_X^2+n)$,  $n\leq 9$.

 Likewise, for $K^2_S=4$ one has $K^2_X\geq -6$ and $n\leq 8$; for $K^2_S=6$ one
  has $K^2_X\geq -4$
 and $n \leq 7$;
for $K^2_S=8$ one has $K^2_X\geq -2$ and $n\leq 6$; finally $K^2_S\neq 10$.

\

\
\textbf{Step 4:} \textit{If $S$ is an  irregular  surface of general type, then $S$ is not of Albanese general type and the genus of a general fibre of the Albanese fibration is $>2$.
}\label{lema5.2}

Suppose that  $q(S)\geq 1$. Then $p_g(S)=q +1 \geq 2$ and, since an irregular  surface satisfies $K^2\geq 2p_g$ by (\cite{debarre}),
$K^2_S\geq 4$.

Since $S$ is a double covering of a surface with $p_g(X)=q(X)=0$
 and $q(S)>0$,  by the De-Franchis theorem (\cite{defra}) $S$ is not of Albanese general type, and so we can consider
the Albanese fibration
$$f:S\to B,$$
where $q(S)$ is the genus of $B$. We
denote by $g$ the genus of a
general fibre of $f$ and write $q:=q(S)$.
We have the following commutative diagram:
$$
\begin{CD}\ S@>f>>B\\ @V\pi VV  @VV \pi' V\\ X@>f' >>\mathbb{P}^1
\end{CD}
$$
 where $\pi'$ is a 2:1 map with $deg R=2q + 2$,
 where $R$ is the ramification divisor.

By the  appendix of \cite{debarre}
$$K_S^2\geq 8(g-1)(q-1).$$

Since, by Step  2, $K_S^2\leq 10$, the only possibilities are   $q=1$, $g\geq 2$ and $deg R=4$,  or $q=2$, $g=2$ and
$deg R=6$.
In this last  case  the slope inequality (\cite{xi2})
         $$
         \frac{4(g-1)}{g}\leq
         \frac{K_S^2-8(g-1)(q-1)}{\mathcal{X}(\mathcal{O}_X)-(g-1)(q-1)}\leq
         12
         $$
yields  $K_S^2=10$ and thus  $6\leq n \leq 7$.

\

Assume in either case  that $g=2$.
Then $f$ and $f'$ are fibrations of genus $2$ with $2q +2$
fibres of $f'$ corresponding to the ramification points of $\pi'$. Denote by
$F_1,...,F_{2q+2}$ those fibres of $f'$ and let $\pi^*(F_i)=2\tilde{F}_i$,
where $\tilde{F}_i$ is a fibre of $f$.
Since a
fibration of genus $2$
 does not have multiple fibres all the fibres $F_j$, $j=1,...,2q+2$, have to
 contain some  of the $(-4)$-curves and  all its other
components will appear with  even multiplicity.  So we can write  $F_j = C_1 +...+C_s+ 2D$, where  $C_1,...,C_s$ are
 $(-4)$-curves in $C$ appearing with odd multiplicity in $F_j$ and $D$ is an effective divisor. Since $K_X(C_1 +...+C_s)=2s$ and $K_XF_j=2$, by the assumption $g=2$, then
$K_XD=1-s$. Now, $C_iF=0$ implies $C_iD=2$, and so from $DF=0$, one obtains
 $2D^2=-2s$, this is  $D^2=-s$. But then $K_XD+D^2=1-2s$, and this contradicts the
adjunction formula.
So we can conclude that $g>2$.

\

\textbf{Step 5:} \textit{If $S$ is irregular, then $q(S)=1$. Also,
 $6\leq K_S^2\leq 10$ and $4\leq n\leq 9$.}\\
 By the previous step $g>2$, whence $q=1$.

Notice that a surface with an Albanese fibration with $g\neq
2$ satisfies $K^2\geq \frac{8}{3}\mathcal{X}(\mathcal{O})$ (\cite{ho3}) and thus $K_S^2\geq 6$.

Finally, as we have seen for the regular case, we have $h^{1,1}(X)\leq
h^{1,1}(S)$ and since $q(S)=1$, one has $b_2(S)=c_2(S) +2$; also,
$b_2(S)=h^{1,1}(S)+4$, then $b_2(X)\leq c_2(S)-2$. Applying this
inequality for $K_S^2=6$,
we obtain $n < 10$.
\end{proof}

\


\bigskip

\bigskip

\begin{minipage}{13cm}
\parbox[t]{6cm}{Mar\'ia Mart\'i S\'anchez\\
Center for Mathematical Analysis, Geometry and Dynamical Systems\\
Instituto Superior T\'ecnico\\
Universidade T{\'e}cnica de Lisboa\\
Av.~Rovisco Pais\\
1049-001 Lisboa, PORTUGAL\\
mmartisanchez@educa.madrid.org
} \hfill
\end{minipage}

\end{document}